\documentclass[12pt,letterpaper]{amsart} 
\usepackage{palatino, euler, epic,eepic, amssymb, xypic, floatflt, microtype}
\usepackage{verbatim,color}

\usepackage{tikz}
\usetikzlibrary{decorations.markings}
\definecolor{Burgundy}{RGB}{144,0,32}
\usepackage{float}
\input xy
\xyoption{all}

\newtheorem{tm}{Theorem}[section]
\newtheorem{pr}[tm]{Proposition}
\newtheorem{lm}[tm]{Lemma}
\newtheorem{co}[tm]{Corollary}
\newtheorem{df}[tm]{Definition}
\newtheorem{rem}[tm]{Remark}

\newcommand{\Q}{{\mathbb Q}}

\newcommand{\z}{{\zeta}}

\newcommand{\x}{{\mathbf{x}}}

\usepackage{colonequals}
\usetikzlibrary{calc}
\usetikzlibrary{decorations.pathreplacing}
\usetikzlibrary{arrows,chains,matrix,positioning,scopes}

\makeatletter
\newcommand{\oset}[3][0ex]{%
  \mathrel{\mathop{#3}\limits^{
    \vbox to#1{\kern-2\ex@
    \hbox{$\scriptstyle#2$}\vss}}}}
\makeatother

\begin{document}

\pagestyle{plain}
\title{Higher Euler-Kronecker Constants of Number fields}
\author{Samprit Ghosh}
\address{Mathematical Sciences 468, University of Calgary, \newline 2500 University Drive NW, Calgary, Alberta, T2N 1N4, Canada. }
\email{samprit.ghosh@ucalgary.ca}
\thanks{}
\begin{abstract}
The higher Euler-Kronecker constants of a number field $K$ are the coefficients in the Laurent series expansion of the logarithmic derivative of the Dedekind zeta function about $s=1$. These coefficients are mysterious and seem to contain a lot of arithmetic information. In this article, we study these coefficients. We prove arithmetic formulas satisfied by them and prove bounds. We generalize certain results of Ihara.
 \end{abstract}
\maketitle
\tableofcontents

\section{Introduction} 
The famous \emph{Euler's constant}, denoted by $\gamma$ is defined as 
\begin{equation*}
	\gamma = \lim_{x \rightarrow \infty} \left( \sum \limits_{n \leq x } \frac{1}{n} \; - \; \log x \right).
\end{equation*}
This constant contains a lot of arithmetic information and appears in several areas of mathematics. Its connection to the Riemann zeta function is perhaps the most profound one, namely,  $\gamma$ appears as the constant term in the Laurent series expansion of the Riemann zeta function, at $s=1$. 
\[
\zeta(s) = \frac{1}{s-1} + \gamma + O (s-1).
\]
For any number field $K$, Ihara in \cite{ihara1} introduced a generalization of $\gamma$  as the constant term appearing in the Laurent series expansion of the logarithmic derivative of the Dedekind zeta function, at $s=1$. 
\[
\frac{\zeta_K'}{\zeta_K}(s) = -\frac{1}{s-1} + \gamma_K + O (s-1).
\]
The constant $\gamma_K$ is called the Euler-Kronecker constant. Ihara then goes on to show interesting formulae satisfied by $\gamma_K$ and deduces important bounds. More recently similar results has also been proved by Dixit and Murty in \cite{anupram}, using more elementary methods.

In this article we study the higher coefficients that appear in the Laurent series expansion of $\frac{\zeta_K'}{\zeta_K}(s)$ about $s=1$. We will refer to these as \emph{higher Euler-Kronecker constants}. Our results join a small but growing number of works which signal a new perspective in number theory, namely, the importance of all the Taylor coefficients of zeta and $L$-functions and not just the leading term. In a different context, this perspective has been most remarkably explored in the recent Annals paper \cite{shtuka} of Yun and Zhang. They show that, at a central critical point, the $n$th Taylor coefficient of an $L$-function for PGL$(2)$ over a function field, has a geometric interpretation for every $n$. Admittedly, the results of this article are in a different domain altogether, however, they share a common element, namely, the importance of all Taylor coefficients. The author hopes that this study will inspire future mathematicians to look at these higher coefficients more closely.
\subsection*{Acknowledgements.} We thank Prof. V. Kumar Murty for his valuable suggestions and encouragement. We thank Prof. David Rohrlich for reading a previous version of this article and sharing valuable feedback. Based on his remarks, several corrections and improvements have been made.

\subsection{Background and Motivation} \label{NFprelim}
Let $K$ be an algebraic number field with degree $[K: \mathbb{Q}] = n_K$.  The Dedekind zeta function of $K$ is defined as 
\begin{equation*}
	\z_K(s) = \sum_{\mathfrak{a}} \frac{1}{(N \mathfrak{a})^s}
\end{equation*}
for Re$(s)>1$, where the sum is taken over all integral ideals $\mathfrak{a}$ of the ring of integers $ \mathcal{O}_K$.
Hecke showed that $(s-1)\z_K(s)$ extends to an entire function. There is a simple pole of $\z_K(s)$ at $s=1$ and the residue satisfies the famous \emph{class number formula}  
$$\lim_{s \rightarrow 1} \; (s-1)\z_K(s) = \dfrac{2^{r_1} (2 \pi)^{r_2} h R}{\omega \sqrt{|d_K|}} \; ,$$
where $r_1$ denotes the number of real embeddings of $K$, $2r_2$ is the number of complex embeddings, $h$ is the class number, $R$ is the regulator, $\omega$ is the number of roots of unity, and $d_K$ is the discriminant of $K$. For the logarithmic derivative one can write  
\begin{equation}\label{eq:vonMangoldtg}
	- \dfrac{\z_K'(s)}{\z_K(s)} = \sum_{\mathfrak{a}} \dfrac{\Lambda(\mathfrak{a})}{(N {\mathfrak{a})^s}} \; .
\end{equation}	
Here $\Lambda(.)$ is the number field analogue of the von Mangoldt function given by 
	\[
	\Lambda(\mathfrak{a}) = \begin{cases} \log \; N\mathfrak{p} \;\; & \text{ if } \mathfrak{a} = \mathfrak{p}^k \; \text{for some prime ideal } \mathfrak{p} \\ 0 & \text{ otherwise.}
	\end{cases}
	\]
	By applying the Tauberian theorem to \eqref{eq:vonMangoldtg}, one can deduce the number field analogue of the prime number theorem, namely the prime ideal theorem.
	$$ \pi_K(x) \sim \frac{x}{\log x} \;\;\; \text{ as } x \rightarrow \infty \; , $$
	where $\pi_K(x)$ denotes the number of prime ideals of $\mathcal{O}_K$ with norm less than or equal to $x$. For details on the above discussion, one may refer to any standard textbook on analytic number theory, for example \cite{dave}, \cite{cassels} or \cite{mur}. \vspace{2mm} \\ The \textbf{generalized Riemann hypothesis} (GRH) states that all non-trivial zeros (i.e. those in the critical strip) of the Dedekind zeta function is on the $s = \frac{1}{2}$ line.\\
	 
	 Consider the completed zeta function
	\begin{equation*}
		\xi_K(s) = s(s-1) 2^{r_2}\left( \frac{\sqrt{|d_K|} }{2^{r_2} \pi^{n_K/2}}\right)^{s} \Gamma\left(\frac{s}{2}\right)^{r_1} \Gamma(s)^{r_2} \zeta_K(s) \; , 
	\end{equation*}
	where  $[K : \Q] = n_K$. In \cite{licoeff}, Li introduced the following sequence of numbers $\{ \lambda_n \}$, now known as Li's coefficients. 
	\begin{equation}
		\lambda_n = \frac{1}{(n-1)!} \frac{d^n}{ds^n} \left(s^{n-1} \log \xi_K(s) \right) \Big|_{s=1} \text{ \hspace{1cm} for } n \geq 1, 
	\end{equation}
	and showed the following theorem. 
	\begin{tm} \textbf{(Li's Criterion)} 
		The general Riemann hypothesis for $\z_K(s)$ holds if and only if $\; \lambda_n$ is non-negative for all $n \geq 1$.
	\end{tm}
	Later Bombieri and Lagarias also gave an alternative proof in \cite{bombli}. Brown in \cite{brownli} proved an effective version of Li's theorem, showing non-negativity  of the first few $\lambda_i$'s, give zero-free regions of a certain shape around $s=1$. In particular he showed, just $\lambda_2 \geq 0$ implies non-existence of the exceptional \emph{Siegel zeros}. We recall, a well-known result of Stark says that $\zeta_{K}(s)$ has at most one zero in the region 
	$$1 - \frac{1}{4\log |d_K|} < \sigma < 1, \;\; |t| \leq \frac{1}{4\log |d_K|},$$ where $s = \sigma + it$. This zero, if it exists, is necessarily real and simple. We call this a Siegel zero.
	\begin{rem}
		Note that, if we write the Laurent series about $s=1$, of the logarithmic derivative of $\z_K(s)$, then $\lambda_2$ involves the constant term and the first coefficient, i.e. the coefficient of $(s-1)$, together with some terms coming from the $\Gamma$- factors . This was our primary motivation to closely study this first coefficient. We later found that many of our results easily generalized to higher coefficients.
	\end{rem}  
	\begin{df}
		Let the Laurent series of the logarithmic derivative of $\z_K(s)$ about $s=1$ be 
		\begin{equation}\label{zeta_log}
			\dfrac{\z_K'(s)}{\z_K(s)} = \dfrac{-1}{s-1}\; + \; \gamma_{K,0} \; + \; \sum_{r=1}^{\infty} \gamma_{K,r}(s-1)^r \;. 
		\end{equation}
	Then	$\gamma_{K,r}$ is called the $\mathbf{r}$\textbf{-th Euler-Kronecker constant}.
	\end{df}
	
	With this notation we have $\gamma_{K,0}=\gamma_K$ of Ihara, whereas $\gamma_{\Q,0} = \gamma$.

	It is worth pointing out that this article does not deal with the subtleties of sign of these coefficients as Li's criterion demands. Instead we present a preliminary study of them. These coefficients are from very local information at $s=1$ and yet somehow they are able to capture what is happening at $s= \frac{1}{2}$, giving us information about all zeros! We think, in future one might be able to deduce zero-free regions and other interesting results from our formulas and bounds.
	
	Before we write our formulas for $\gamma_{K,r}$, we note some similar formulas for $\zeta(s)$ and $\zeta_K(s)$ that already exists in the literature. Suppose we write 
	\begin{equation*}
		\z(s) = \frac{1}{s-1} + \gamma + \sum_{n=1}^{\infty} \mathfrak{s}_n (s-1)^n   .
	\end{equation*}
	In 1885, T. J. Stieltjes \cite{stiljets} showed that 
	\begin{equation*}
		\mathfrak{s}_n = \frac{(-1)^n}{n!} \lim_{x \rightarrow \infty } \left( \sum_{m=1}^{x} \frac{(\log m)^n}{m} - \frac{(\log x)^{n+1}}{n+1}\right).
	\end{equation*}
	These $\mathfrak{s}_n$ are called the Stieltjes constants, the generalized Euler constants or sometimes the Euler-Stieltjes constants. For the Dedekind zeta function, let us write 
	\begin{equation*}
		\z_K(s) = \sum_{n=-1}^{\infty} \mathfrak{s}_{K,n} (s-1)^n   
	\end{equation*}
	The author found a similar formula in a much recent paper and does not know if similar formula has been written down in the past. The following is due to Eddin \cite[Theorem 2]{saad}. 
	
	\begin{equation*}
		\mathfrak{s}_{K,n} = \frac{(-1)^n}{n!} \lim_{x \rightarrow \infty } \left( \sum_{N \mathfrak{a} \leq x} \frac{(\log N \mathfrak{a})^n}{N \mathfrak{a}} - 	\mathfrak{s}_{K,-1} \frac{(\log x)^{n+1}}{n+1}\right) \text{\hspace{0.5cm} for } n \geq 1,
	\end{equation*}
	and 
	\begin{equation*}
		\mathfrak{s}_{K,0} = \lim_{x \rightarrow \infty } \left( \sum_{N \mathfrak{a} \leq x} \frac{1}{N \mathfrak{a}} - 	\mathfrak{s}_{K,-1} \log x \right) + \; \mathfrak{s}_{K,-1}.
	\end{equation*}
	The following formula for the logarithmic derivative of the Riemann zeta function is also known. Let 
	\begin{equation*}
		\frac{\z'(s)}{\z(s)} = \frac{-1}{(s-1)} + \sum_{n=0}^{\infty} \gamma_n (s-1)^n  \;, 
	\end{equation*}
	then
	\begin{equation*}
		\gamma_n = \frac{(-1)^{n+1}}{n!} \lim_{x \rightarrow \infty } \left( \sum_{m<x} \frac{\Lambda(m) (\log m)^n}{m} - \frac{(\log x)^{n+1}}{n+1}\right).
	\end{equation*}
	For a proof see \cite{titch1}. The author has not found a similar formula for Dedekind zeta functions and so, is recording it as a theorem below. 
	\subsection{Statement of main results}	
\begin{tm}\label{thm:EKlimformula1}
	Let $K$ be an algebraic number field and write the Dirichlet series
	$$ - \frac{\zeta_K'(s)}{\zeta_{K}(s)} = \sum_{n=1}^{\infty} \frac{\Lambda_K(n)}{n^s}. $$
	For $r \geq 0$, the $r$-th Euler-Kronecker constant of $K$ satisfies the  arithmetic formula 
	\begin{equation}
		\gamma_{K,r} = \frac{(-1)^{r+1}}{r!} \lim_{x \rightarrow \infty} \left( \sum_{n \leq x} \frac{\Lambda_K(n) (\log n)^r}{n} - \frac{(\log x)^{r+1}}{r+1}  \right).
	\end{equation}
\end{tm}
We  prove the following bounds. In the next theorem, we are using the notation $\log_m(x)$ to denote $\log_m(x) = \log ( \log( \cdots(\log x)))$.
\begin{tm}\label{thm:EKbound}
Assuming GRH, for $r \geq 1$, odd, we have
	\begin{equation}\label{boundodd}
		 \gamma_{K,r} \geq \frac{2^{r+1}}{r!} (\log r + \log_2 |d_K| + 2 \log_3 |d_K|)^{r+1} \left( \frac{-1}{r+1} + O  \left(\frac{\log r + \log_3 |d_k|}{\log_2 |d_K|}\right) \right).
	\end{equation}
	Whereas for $r \geq 2$, even, 
		\begin{equation}\label{boundeven}
		\gamma_{K,r} \leq \frac{2^{r+1}}{r!} (\log r + \log_2 |d_K| + 2 \log_3 |d_K|)^{r+1} \left( \frac{1}{r+1} + O  \left(\frac{\log r + \log_3 |d_k|}{\log_2 |d_K|}\right) \right).
	\end{equation}
\end{tm}
In Section \ref{sec:gen_uncond_bound}, Theorem \ref{uncond_Bound} we have also noted down general unconditional bounds. 
\begin{rem}
	As one may notice, these bounds are good for smaller values of $r$. As $r$ grows, our bound seems to be more reliant on $r$ than the number field.
\end{rem}

\begin{rem} To obtain the above formulae and bounds we didn't use the heavy machinery employed by Ihara. However, we were able to generalize Ihara's methods for higher coefficients as well and got the formula stated below. Unfortunately, these generalizations did not yield better bounds, so we have only included the arithmetic formula in this article. This formula below is important as unlike Theorem \ref{thm:EKlimformula1}, it does not require the explicit knowledge of the coefficients $\Lambda_K(n)$ of the Dirichlet series,  to do computations.
\end{rem}
\begin{tm}\label{thm:EKgenIhara}
For $x>1$ and $r \geq 0$, consider the function
$$\Phi_K(r, x) = \dfrac{1}{x-1} \sum \limits_{k, \; N(P)^k \leq \; x} \left( \dfrac{x}{N(P)^k} -1 \right) k^r (\log N(P))^{r+1}.$$
Let $f(r,x)$ be a recursively defined function as follows.  
\begin{align*}
	f(0,x) &=  \log x,\\
	f(1,x) &=   \left[ 2+  \frac{1+x}{1-x} \log x +\frac{1}{2}   (\log x )^2 \right], \;\; \text{ and for }r \geq 2  \\
	f(r,x) &= \frac{1}{r+1} (\log x)^{r+1} \; + \; \frac{1+x}{1-x}(\log x)^{r} + r(r-1)f(r-2,x).
\end{align*}
Then we have 
	\begin{equation}
	\gamma_{K,r} + (-1)^{r} = \frac{(-1)^{r+1}}{r!} \lim_{x \rightarrow \infty} \left( \frac{}{} \Phi_K(r,x) - f(r,x) \right).  
	\end{equation}	
\end{tm}
\begin{rem}
	The $r=0$ case was proven by Ihara in \cite[Corollary 1, (1.6.1)]{ihara1}. In our proof in Section \ref{sec:Ihara}, we will focus on $r \geq 1$.
\end{rem}
\section{Auxiliary results}
We will now look at a few general results pertaining to a function defined by a  Dirichlet series and its Laurent series expansion.  The next Proposition can be considered as a generalization of \cite[Proposition 2.1]{anupram}.  These results can be deduced from the short paper \cite{briggs} of Briggs and Buschman. \\

Let $L(s)$ be a Dirichlet series such that for some $C>0$, we can write the following Laurent
series at $s=1$ 
$$ - \frac{L'(s)}{L(s)} = \sum_{n=1}^{\infty} \frac{b_n}{n^s} = \frac{C}{s-1} + C_0 + \sum_{r=1}^{\infty} \frac{(-1)^r C_r}{r!} (s-1)^r .  $$
Let $B(x) := \sum \limits_{n \leq x} b_n$, and assume that $E(x) := B(x) - Cx = O(x^b)$ where $0 \leq b<1$. By partial summation we have 
$$ - \frac{L'(s)}{L(s)}  = s \int_{1}^{\infty} \frac{B(t)}{t^{s+1}} dt = s \int_{1}^{\infty}  \frac{Ct + E(t)}{t^{s+1}}dt = \frac{Cs}{s-1}+s  \int_{1}^{\infty}  \frac{E(t)}{t^{s+1}}dt. $$
Note that the integral is analytic for Re$(s) > b$. Let us define 
$$L_1(s) := s  \int_{1}^{\infty}  \frac{E(t)}{t^{s+1}}dt = - \frac{L'(s)}{L(s)} -  \frac{Cs}{s-1},  $$ 
and so,
\begin{equation}\label{defL_1}
	L_1(s) = -C + C_0 + \sum_{r=1}^{\infty} \frac{(-1)^r C_r}{r!} (s-1)^r . 
\end{equation} 
\begin{pr}\label{mainprop}
	If $u < -b$ and $L(s), \; L_1(s)$ as above, then for $x \geq 1$ we have 
	\begin{align*}
		\sum_{n \leq x} n^u b_n \log^r(n) =(-1)^rL_1^{(r)}(-u) &+ C \int_{1}^{x}t^u \log^r t \; dt \; + D_r + E(x)x^u \log^r(x) \\ & + \int_{x}^{\infty} (u \log t+r )  E(t)  t^{u-1} \log^{r-1} t \; dt 
	\end{align*}
	Where $D_0 = C$ and $D_r = 0$ for $r\geq 1$. 
\end{pr}
\begin{proof}
	From Abel's summation formula we have 
	\begin{equation}\label{Abel1}
		\sum_{n \leq x} n^u b_n \log^r(n) = B(x)x^{u}\log^r(x) - \int_{1}^{x} B(t) \frac{d}{dt} (t^u \log^rt)dt.
	\end{equation}
	We note the following interesting relation coming from elementary calculus using the Leibniz product rule for higher derivatives.
	\begin{align*}
		\frac{d^r}{du^r} (ut^{u-1}) = \sum_{k=0}^r \binom{r}{k} \; \frac{d^{k}}{du^k} u \; \frac{d^{r-k}}{du^{r-k}} t^{u-1} 
		=ut^{u-1}\log^r t + rt^{u-1}\log^{r-1} t = \frac{d}{dt} (t^u \log^rt)
	\end{align*} 
	Hence, substituting this in \eqref{Abel1} 
	\begin{align}\label{Abel3}
		\nonumber	\sum_{n \leq x} n^u b_n \log^r(n) &= B(x)x^{u}\log^r(x) - \int_{1}^{x} B(t) \frac{d^r}{du^r} (ut^{u-1})dt \\
		&= Cx^{u+1}\log^r x + E(x)x^u\log^r x - \frac{d^r}{du^r} \left[ \int_{1}^{x}  (Ct+E(t))ut^{u-1} dt \right]. 
	\end{align}
	Focusing on the last integral 
	\begin{align*}
	- \frac{d^r}{du^r} \left[ \int_{1}^{x}  (Ct+E(t))ut^{u-1} dt \right] 
		&= \frac{d^r}{du^r} \left[ -Cu \int_{1}^{x} t^u dt - u \int_{1}^{\infty} t^{u-1}E(t) dt + u \int_{x}^{\infty}t^{u-1}E(t)dt  \right] \\
		&= \frac{d^r}{du^r} \left[ -Cu \int_{1}^{x} t^u dt +L_1(-u) + u \int_{x}^{\infty}t^{u-1}E(t)dt  \right]\\
		&= (-1)^rL_{1}^{(r)}(-u) + \frac{d^r}{du^r} \left[ -Cu \int_{1}^{x} t^u dt + u \int_{x}^{\infty}t^{u-1}E(t)dt  \right]
	\end{align*}
	We now combine the first term in the brackets above with the first term of \eqref{Abel3}.
	\begin{align}
		\nonumber Cx^{u+1}\log^r x   - C \frac{d^r}{du^r} u \int_{1}^{x} t^u dt \
		&= C \frac{d^r}{du^r} \left[ x^{u+1} - 1 - u\int_{1}^{x} t^u dt \right] + \frac{d^r}{du^r} C \\ \nonumber
		&= C \frac{d^r}{du^r} \left[  \int_{1}^x t^u dt \right] +D_r \\
		&= C\int_{1}^{x} t^u \log^r t \; dt \; + D_r.
	\end{align}
	Recall, as defined in the statement, $D_0 = C$ and $D_r = 0$ for $r \geq 1$. Lastly,
	\begin{align}
		\nonumber \frac{d^r}{du^r} \left[ u \int_{x}^{\infty}t^{u-1}E(t)dt \right] &=  \int_{x}^{\infty} (u \log t+r )  E(t)  t^{u-1} \log^{r-1} t \; dt 
	\end{align}
	gives us our desired result.
\end{proof}
\begin{pr}\label{mainlimprop}
	For $x \geq 1$ and $r \geq 0$ we have 
	\begin{equation} 
		C_r =  \sum_{n \leq x} \frac{b_n \log^r n}{n} - \frac{C}{r+1} \log^{r+1} x  - \frac{E(x)\log^r x}{x}  - \int_{x}^{\infty}  \frac{(\log^r t + r \log^{r-1} t) E(t)}{t^2} dt 
	\end{equation}
	and so, letting $x \rightarrow \infty$ we get the formula
	\begin{equation}
		C_r = \lim_{x \rightarrow \infty} \left(  \sum_{n \leq x} \frac{b_n \log^r n}{n} - \frac{C}{r+1} \log^{r+1} x \right).
	\end{equation}
\end{pr}
\begin{proof}
	Note that, from \eqref{Abel1} we get $L_1(1) = -C+C_0$, whereas for $r \geq 1$, we have $L_1^{(r)}(1) = (-1)^r C_r$. Substituting this and putting $u=-1$ in Proposition \ref{mainprop} we get, for $r=0$ 
	\begin{align*}
		& \sum_{n \leq x} \frac{b_n}{n} = -C+C_0 + C \int_{1}^x \frac{dt}{t} + C  + \frac{E(x)}{x} - \int_{x}^{\infty} \frac{E(t)}{t^2} dt \\
		\text{and so,}\;\;\;  & C_0 =  \sum_{n \leq x} \frac{b_n}{n} - C \log x + \int_{x}^{\infty} \frac{E(t)}{t^2} dt + O(x^{b-1}).
	\end{align*}
	Similarly, for $r \geq 1$, we have
	\begin{align*}
		& \sum_{n \leq x} \frac{b_n \log^r n}{n} = C_r + C \int_{1}^x \frac{\log^r t}{t} dt  - \int_{x}^{\infty}  \frac{(\log^r t + r \log^{r-1} t) E(t)}{t^2} dt + O(x^{b-1}\log^r x)\\
		& C_r =  \sum_{n \leq x} \frac{b_n \log^r n}{n} - \frac{C}{r+1} \log^{r+1} x + \int_{x}^{\infty}  \frac{(\log^r t + r \log^{r-1} t) E(t)}{t^2} dt + O(x^{b-1}\log^r x).
	\end{align*}
	Letting $x \rightarrow \infty$ we get the other formula.
\end{proof}

\section{Proof of  Theorem \ref{thm:EKlimformula1} and Theorem \ref{thm:EKbound}}
We will now apply these two propositions to derive the arithmetic formula and bounds satisfied by $\gamma_{K,r}$. \vspace{2mm}\\
\textbf{\emph{Proof of Theorem \ref{thm:EKlimformula1}}}. Follows directly from Proposition \ref{mainlimprop}, by noting that for the Dedekind zeta function, $L(s) = \zeta_K(s)$, we have $C=1$, $b_n = \Lambda_K(n)$ and $\gamma_{K,r} = \frac{(-1)^{r+1}}{r!}C_r$. Note that there is an extra negative sign, coming from the fact that $\gamma_{K,r}$ are coefficients of the Laurent series of $ \frac{\zeta_{K}'(s)}{\zeta_K(s)}$, whereas, $\Lambda_K(n)$ are coefficients of the Dirichlet series of $-\frac{\zeta_{K}'(s)}{\zeta_K(s)}$.  \qed

\begin{co}
	For $r \geq 1$ we have
	$$ \gamma_{K,r} = \frac{(-1)^{r}}{r!} \int_{1}^{\infty} \frac{(\log^r(t) + r \log^{r-1}t) \Delta_K(t)}{t^2} dt \;, $$
	where $\Delta_K(x) = \sum_{n \leq x} \Lambda_K(n) - x$ and $\Lambda_K(n)$ as in Theorem \ref{thm:EKlimformula1} .    
\end{co}
\begin{proof}
	This is obtained by simply taking $x=1$ in Proposition \ref{mainlimprop}.
\end{proof}
\begin{rem}
	This can be seen as the generalization of the following integral formula (see \cite[Theorem 1.1]{anupram}) that relates the Euler-Kronecker constant $\gamma_{K}$ (in our notation $\gamma_{K,0}$) to the error term in the prime ideal theorem.
	$$\gamma_{K,0} = 1 - \int_{1}^{\infty} \frac{\Delta_K(t)}{t^2} dt.$$

\end{rem}

We will now give a proof for the bounds. \vspace{2mm}\\ 
\textbf{\emph{Proof of Theorem \ref{thm:EKbound}.}} Adapting Proposition \ref{mainlimprop} for $\frac{\zeta_K'(s)}{\zeta_K(s)}$ we have 
\begin{align*}
(-1)^{r+1} r! \cdot 	\gamma_{K,r}  =   \sum_{n \leq x} \frac{\Lambda_K(n) \log^r n}{n} & - \frac{1}{r+1} \log^{r+1} x - \frac{\Delta_K(x) \log^r x}{x}\\ 
	& - \int_{x}^{\infty}  \frac{(\log^r t + r \log^{r-1} t) \Delta_K(t)}{t^2} dt \\
\end{align*}
Assuming GRH, we know that (see \cite[Theorem 4]{serreprime})
\begin{equation}
	\Delta_K(x) \ll \sqrt{x} \left(  \log |d_K| + n_K \log x \right) \log x.
\end{equation}
(Recall that we are writing $n_K = [K : \mathbb{Q}]$.)\\

Thus when $r\geq 1$ is odd, from positivity  of $\Lambda_K(n)$'s and using the above bound for $\Delta_K(n)$ we have
\begin{align*}
	r! \cdot \gamma_{K,r} & \geq - \frac{1}{r+1} \log^{r+1} x  + O \left( \frac{r( \log |d_K| + n_K \log x) \log^{r+1} x}{\sqrt{x}} \right) \\
	& = \log^{r+1} x  \left( \frac{-1}{r+1} + O \left( \frac{r( \log |d_K| + n_K \log x) }{\sqrt{x}} \right)  \right)
\end{align*}
Thus choosing $x = r^2( \log |d_K|)^2 ( \log_2 |d_K|)^4$  and using the fact that $\frac{n_K}{\log |d_K|}$ is bounded, due to Minkowski, we get  
\begin{align*}
	\gamma_{K,r} &\geq  \frac{2^{r+1}}{r!} \left( \log r +  \log_2 |d_K| + 2  \log_3 |d_K| \right)^{r+1} \left( \frac{-1}{r+1} + O \left( \frac{ \log r +  \log_3 |d_K| }{ \log_2 |d_K|} \right) \right)
\end{align*}
On the other hand, for $r \geq 2$ even, we follow the same process while being mindful of the negative signs. \qed
\section{General unconditional bounds}\label{sec:gen_uncond_bound} 
Recall that we're writing the Laurent series of $\frac{\z_K'(s)}{\z_K(s)}$ about $s=1$ as
\begin{align*}
	\frac{\z_K'(s)}{\z_K(s)} + \frac{1}{s-1} = \gamma_{K,0} + \sum_{r=1}^{\infty} \gamma_{K,r} \; (s-1)^r.
\end{align*}
Differentiating both sides $r$ times and letting $s \rightarrow 1$ we get, 
\begin{equation}\label{gamma_k,r}
	\lim_{s \rightarrow 1^+} \;  \dfrac{d^r}{ds^r} \left[ \dfrac{\z_K'(s)}{\z_K(s)} +  \dfrac{1}{(s-1)}\right] = r! \cdot \gamma_{K,r}
\end{equation}
On the other hand, from the Hadamard factorization, we have (e.g. see \cite[Lemma 3]{stark1}) 
\begin{equation}\label{stark1}
	-  \dfrac{\z_K'(s)}{\z_K(s)}  = \dfrac{1}{s} + \dfrac{1}{s-1} - \sum \dfrac{1}{s - \rho}+ \frac{1}{2} \log |d_K| + \beta_K + \tilde{ \Gamma}_K (s)
\end{equation}
where the sum, $\rho$ runs over all non-trivial zeros of $\z_K(s)$, counted with multiplicities. Here 
\begin{align*}
	\beta_K &= -\left\{\frac{r_1}{2}(\gamma + \log 4 \pi) + r_2 (\gamma + \log 2 \pi)\right\},  \text{ and \hspace{2cm}}\; \\
	\tilde{ \Gamma}_K (s) &= \frac{r_1}{2} \left(\frac{\Gamma'}{\Gamma} \left(\frac{s}{2}\right) - \frac{\Gamma'}{\Gamma} \left(\frac{1}{2}\right)\right) + r_2 \left(\frac{\Gamma'}{\Gamma}(s) - \frac{\Gamma'}{\Gamma}(1)\right). \;\; 
\end{align*}  
with $r_1$, $r_2$ respectively being the number of real, imaginary places of $K$. Thus rearranging the terms in \eqref{stark1} and taking $\lim s \rightarrow 1$, we get
\begin{align} \label{alphak}
	\gamma_{K,0} + 1 &= \sum_{\rho} \dfrac{1}{1- \rho} - \frac{1}{2} \log |d_K| - \beta_K.  
\end{align}
On the other hand, differentiating \eqref{stark1} $r$ times  we get, 
\begin{equation}\label{stark1der}
	\dfrac{d^r}{ds^r} \left[ \dfrac{\z_K'(s)}{\z_K(s)} +  \dfrac{1}{(s-1)}\right] =  \frac{(-1)^{r+1} r!}{s^{r+1}} + \sum \dfrac{(-1)^r r!}{(s - \rho)^{r+1}} - \tilde{ \Gamma}^{(r)}_K(s). 
\end{equation}
Letting $s \rightarrow1$ and using \eqref{gamma_k,r} we derive 
\begin{equation}\label{gamma_kr}
	\gamma_{K,r} = (-1)^{r+1} + \sum \dfrac{(-1)^r }{(1 - \rho)^{r+1}} -  \frac{1}{r!} \tilde{ \Gamma}^{(r)}_K(1). 
\end{equation}
For the $\tilde{ \Gamma}_K$ term, we look at the series expansion of the \emph{digamma} function $\psi(s) = \frac{\Gamma'}{\Gamma}(s)$ and its derivatives. 
\begin{align*}\label{gammaformula1}
		\psi(s) = \frac{\Gamma'}{\Gamma}(s) = - \gamma - \sum_{k=0}^{\infty} \left( \frac{1}{s+k} - \frac{1}{1+k} \right)
	\Rightarrow \; \psi^{(r)}(s)  = (-1)^{r+1}  \sum_{k=0}^{\infty} \frac{r!}{(s+k)^{r+1}}.
\end{align*} 
Note that $\tilde{\Gamma}_K(s) = \frac{r_1}{2} \left[\psi \left(\frac{s}{2}\right) - \psi \left(\frac{1}{2}\right) \right] + r_2 \left[ \psi(s) - \psi (1) \right]$ and so
\begin{equation}\label{gammaformula2}
	\tilde{\Gamma}_K^{(r)}(s) = \frac{r_1}{2^{r+1}} \; \psi^{(r)} \left(\frac{s}{2}\right) + r_2  \; \psi^{(r)}(s).
\end{equation}
Therefore the $\frac{1}{r!} \tilde{ \Gamma}^{(r)}_K(1)$ term in \eqref{gamma_kr} is 
$ O(n_K)$, where $n_K = r_1 + 2r_2 = [K: \Q]$. We now focus on the sum over the non-trivial zeros. This is where the main contribution comes from, in fact, that too from the low hanging zeros. The sum can be split into two parts, those with $|\Im(\rho)| \leq 1$ and $|\Im(\rho)|>1$. \\

Writing $\rho = \sigma + it$ we note that, $|1-\rho| \geq |t|$ and so for $|t|>1$, (and for $r \geq 1$)
\begin{align*}
	 \sum_{|\Im(\rho)|>1} \left| \dfrac{(-1)^r }{(1 - \rho)^{r+1}} \right| & \leq \sum_{|t|>1} \frac{1}{t^2} \leq \sum_{j>1} \sum_{j < |t| < j+1} \frac{1}{j^2} \\
	 & \leq \sum_{j>1} \frac{c_1(n_K \log j + \log |d_K|)}{j^2} = O(\log |d_K|)
	 \end{align*}
This last inequality follows from the well known result below, on the number of zeroes of the Dedekind zeta function (see \cite[Theorem 5.31]{iwaniec2021analytic}).
\begin{lm}\label{numZero}
	Let $N_K(t)$ be the number of zeros of $\zeta_K(s)$ in the region $0 \leq \Re(s) \leq 1$ and $|\Im(s)| \leq t$, then 
	$$N_{K}(t+1) - N_K(t) \ll n_K \log t + \log |d_K|$$ 
	where $n_K = [K: \Q]$ and the implied constant here is absolute.
\end{lm}
Now to compute the contribution from the $|\Im(\rho)| \leq 1$ term, we first recall a result of Stark \cite{stark1}. 
\begin{lm}
	If $K \neq \Q$, $\zeta_{K}(s)$ has at most one zero in the region
	$$ 1 -  \frac{1}{4\log | d_K|} < \Re(s) < 1, \;\;\; |\Im(s)| \leq \frac{1}{4 \log |d_k|}.$$
	If such a zero exists, it is necessarily real and simple. We refer to such a zero by Siegel zero.   
\end{lm}
Looking at the above lemma, we further subdivide the sum as follows (as before writing $\rho = \sigma + it$).
\begin{align*}
	 \sum_{|t|\leq 1} \left| \dfrac{(-1)^r }{(1 - \rho)^{r+1}} \right| & \leq \sum_{|t|\leq \frac{1}{4 \log |d_K|}}  \dfrac{1 }{|(1 - \rho)|^{r+1}} + \sum_{ \frac{1}{4 \log |d_K|} < |t| \leq 1}  \dfrac{1 }{|(1 - \rho)|^{r+1}}   \\
\end{align*} 
For the second sum, note that $|1-\rho| \geq |t| > \frac{1}{4 \log |d_K|}$ and therefore is $O((4\log |d_K|)^{r+2})$ (note that by Lemma \ref{numZero} the number of total terms in the sum is $O(\log|d_K|)$). 
If a Siegel zero does not exists, then for the first sum we also have $|1-\rho| \geq 1-\sigma \geq \frac{1}{4 \log |d_K|}$. Whereas, if a Siegel zero exists, say $\beta_0$, then we will have an extra term $\frac{(-1)^r}{(1-\beta_0)^{r+1}}$. Combining these we obtain the following theorem.
\begin{tm}\label{uncond_Bound}
	For any number field $K \neq \mathbb{Q}$ and $r \geq 1$, if $\zeta_K(s)$ has no  Siegel zero, then 
	$$\gamma_{K,r} = O((4\log |d_K|)^{r+2}) $$
	If $\zeta_K(s)$ has a Siegel zero $\beta_0$, then 
	$$\gamma_{K,r} = \frac{(-1)^r}{(1-\beta_0)^{r+1}} +  O((4\log |d_K|)^{r+2}).$$
	Here the implied constants are absolute.
\end{tm}
\begin{rem}
	We will not delve too much into the various results concerning bounds on $\beta_0$ and leave it to the reader to substitute their favorite one. However, it is worth mentioning that by \cite[Lemma 8]{stark1}, $\beta_0$ arises from a quadratic subfield $F \subseteq K$. If the discriminant of $F$ is denoted by $d_F$, then Siegel's bound gives $1 - \beta_0 \gg |d_F|^{-\epsilon}$. Using the fact that $|d_F|^{n_K/2}$ divides $d_K$, we get that $$\gamma_{K,r} = O\left( |d_K|^{\frac{2(r+2)\epsilon}{n_K}} \right)$$
	One also notes that by the same \cite[Lemma 8]{stark1} of Stark, if $K \neq \Q$, is a normal extension that does not contain any quadratic subfield, then $\beta_0$ does not exists.  
	\end{rem}
	\begin{rem}
		We also note that for $r = 0$, Ihara in \cite{ihara1} showed the bounds 
		\begin{align*}
			\gamma_{K,0} & \leq 2 \log(\log \sqrt{|d_K|}) \; \text{ (under GRH), and} \\
			& \geq - \log \sqrt{|d_K|} \; \text{ \hspace{0.7cm} (unconditionally).}
		\end{align*}
	In our computations we see that this disparity in the size of the lower and upper bounds continue for the higher coefficients as well. The only difference is that it alternates with the parity of $r$, being even or odd, as shown in Theorem \ref{thm:EKbound}. 	
	\end{rem}
\section{Generalizing Ihara's Methods}\label{sec:Ihara} 
 In \cite[Corollary 1]{ihara1} Ihara shows, for the function
\[ \Phi_K(x) = \frac{1}{x-1} \sum_{N(P)^k \leq x} \left( \frac{x}{N(P)^k } - 1 \right) \log N(P) \;\;\;\;\;\;\; (x>1)
 \]
 where $(P, k)$ runs over the pairs of (non-archimedean) primes $P$ of $K$ and positive integers $k$ such that $N(P)^k \leq x$, 
 \begin{equation}\label{Ihara_lim_form}
 	\lim_{x \rightarrow \infty}  (\log x - \Phi_{K}(x)) = \gamma_{K} + 1.
 \end{equation}
 
In this section we will generalize the method employed by Ihara for proving similar limit formulas for higher constants $\gamma_{K,r}$. Let us write, for brevity of notation, 
$$Z_K(s) = -  \dfrac{\z_K'(s)}{\z_K(s)} \;\; \text{ and }\;\; Z_K^{(r)}(s) = - \frac{d^r}{ds^r} \dfrac{\z_K'(s)}{\z_K(s)} $$
To deduce these formulas, we will evaluate the integral 
$$\Psi_K (\mu, r, x) = \dfrac{1}{2\pi i}\int_{c-i\infty}^{c+i \infty} \dfrac{x^{s-\mu}}{s-\mu}Z_K^{(r)}(s) \;ds \; , \;\;\;\; \text{ where } c \gg 1 ,$$ 
 for $\mu = 0$ and $1$, using two different expressions of $Z_K(s)$. One expression comes from \eqref{stark1der} and another from the Euler product 
 $$Z_K(s) =  - \frac{\z_K'(s)}{\z_K(s)} = \sum \limits_{P, k\geq 1} \dfrac{\log N(P)}{N(P)^{ks}}.$$ 
 Differentiating the above $r$ times yields 
 \begin{equation}\label{dirichletZm'}
 	Z_K^{(r)}(s)= \sum\limits_{P, \; k\geq 1} \dfrac{(-1)^r k^r (\log N(P))^{r+1}}{N(P)^{ks}}.
 \end{equation}
 The following classical formulas will come in handy. 
 \begin{equation} \label{classical11}
 	\dfrac{1}{2\pi i}\int_{c-i\infty}^{c+i \infty} \dfrac{y^s}{s} \; ds = 
 	\begin{cases}
 		0 \;\;\; & 0<y<1 \\
 		\frac{1}{2}\;\;\; & y=1 \\
 		1 \;\;\; & y>1.
 	\end{cases}
 \end{equation}
 For $n \geq 1$ 
 \begin{equation}\label{classical22}
 	\dfrac{1}{2\pi i}\int_{c-i\infty}^{c+i \infty} \dfrac{y^s}{s^{n+1}} \; ds = 
 	\begin{cases}
 		0 \;\;\; & 0<y \leq 1 \\
 		\frac{1}{n!} (\log y)^n\;\;\; & y>1.
 	\end{cases}
 \end{equation}
 \begin{lm} We have
 	\[ x\Psi_K(1,r,x) - \Psi_K(0,r,x) =  \sum \limits_{k, \; N(P)^k< \; x} \left( \dfrac{x}{N(P)^k} -1 \right) (-1)^r k^r (\log N(P))^{r+1}. \]
 \end{lm}
 \begin{proof} Using \eqref{dirichletZm'} we have
 	\begin{align*}
 		& \;\;\;\; x\Psi_K(1,r,x) - \Psi_K(0,r,x) \\[1.3ex] 
 		&= x \cdot \sum \limits_{P, \; k\geq 1} \dfrac{(-1)^r k^r (\log N(P))^{r+1}}{N(P)^k} \left[ \dfrac{1}{2 \pi i } \int_{c-i\infty}^{c+\infty} \dfrac{1}{s-1} \left( \dfrac{x}{N(P)^k}\right)^{s-1} \; ds \right] \\[1.5ex] 
 		& \;\;\;\;\;\; - \sum \limits_{P, k\geq 1} (-1)^r k^r (\log N(P))^{r+1} \left[ \dfrac{1}{2 \pi i } \int_{c-i\infty}^{c+\infty} \dfrac{1}{s} \left( \dfrac{x}{N(P)^k}\right)^{s} \; ds \right] \\[1.5ex]
 		&= x \left[  \sum \limits_{k, \; N(P)^k < x} \dfrac{(-1)^r k^r (\log N(P))^{r+1}}{N(P)^k} + \sum_{k, \; N(P)^k = x}\dfrac{(-1)^r k^r (\log N(P))^{r+1}}{N(P)^k} \cdot \frac{1}{2} \right]   \\[1.5ex]
 		& \;\;\;\;\;\;\;\;\; - \left[  \sum \limits_{k, N(P)^k<x } (-1)^r k^r (\log N(P))^{r+1} + \sum \limits_{k, N(P)^k =x} (-1)^r k^r (\log N(P))^{r+1} \cdot \frac{1}{2} \right] \\[1.5ex]
 		&=  \sum \limits_{k, \; N(P)^k< \; x} \left( \dfrac{x}{N(P)^k} -1 \right) (-1)^r k^r (\log N(P))^{r+1} \\
 	\end{align*}
 	Note that the second equality follows from \eqref{classical11}. Also, the second sums in the square brackets cancel each other as $x=N(P)^k$ and there is an $x$ at the front of the first term. 
 \end{proof}
 \begin{df}
 	We define the function, for $x>1$ and $r \geq 0$,
 	$$\Phi_K(r, x) = \dfrac{1}{x-1} \sum \limits_{k, \; N(P)^k \leq \; x} \left( \dfrac{x}{N(P)^k} -1 \right) k^r (\log N(P))^{r+1} $$
 \end{df}
 \begin{rem}
 	We note that $\Phi_K(r, x)$ is the generalization of Ihara's $\Phi_K(x)$ mentioned in the beginning of this section. We note that $\Phi_K(r,x)$ is non-negative, everywhere continuous and monotonically increasing. The points of the form $x=N(P)^k$ are of main interest.  If we call these {critical points}, then $\Phi_K(r,x)$ remains $0$ until the first critical point. Then monotonically increases till the next critical point.  At the next critical point $\Phi_K(r,x)$ acquires new summands but their values are 0 at this point, which makes it continuous. We also note, $$\Phi_K(r,x) \leq (\log x)^r \;  \Phi_K(0,x) \leq n_K (\log x)^r \; \Phi_{\mathbb{Q}}(0,x) \leq n_K (\log x)^{r+1} $$
 The first two inequalities are trivial. For the last one we note that, both $\log x$ and $\Phi_{\mathbb{Q}}(0,x)$ are monotonically increasing, thus it is sufficient to prove the inequality at integer points. We'll in fact show $\Phi_{\Q}(0, n+1) < \log n$.  To see this, we first recall that the exponent of a prime $p$ in the factorization of $(n!)$ is given by $\left\lfloor  \frac{n}{p}\right\rfloor + \left\lfloor  \frac{n}{p^2}\right\rfloor + \cdots + \left\lfloor  \frac{n}{p^k}\right\rfloor $, where $k$ is the highest exponent such that $p^k \leq n$. Therefore \begin{align*}
 	 \log (n!)  &= \sum_{p^k \leq n} \left\lfloor  \frac{n}{p^k}\right\rfloor  \log p  \geq \sum_{p^k \leq n} \left(\frac{n+1}{p^k} - 1 \right) \log p = n \Phi_{\mathbb{Q}}(0, n+1),\\
 	\text{and so, } \;  \Phi_{\mathbb{Q}}(0, n+1) & \leq \frac{\log(n!)}{n} \leq \log n.
 \end{align*} 
 \end{rem}
  We note that 
  \begin{equation}\label{PhiEuler}
  	x\Psi_K(1,r,x) - \Psi_K(0,r,x) = (x-1) (-1)^r \Phi_K(r,x).
  \end{equation}
  Now let us compute $	x\Psi_K(1,r,x) - \Psi_K(0,r,x)$ using the expression coming from \eqref{stark1der}, namely,
  \begin{equation*}
	Z_K^{(r)}(s) = \frac{(-1)^r r!}{(s-1)^{r+1}}+ \frac{(-1)^r r!}{s^{r+1}}  +  \sum \frac{(-1)^{r+1} r!}{(s-\rho)^{r+1}} + \tilde{ \Gamma}^{(r)}_K(s). \;\;\;\;
  \end{equation*}
  We will consider contribution from each of these sums separately and record them as lemmas. Let us denote the contribution from the terms $\frac{(-1)^r r!}{(s-1)^{r+1}}+ \frac{(-1)^r r!}{s^{r+1}}$ by $\mathfrak{F}(r,x)$, i.e.
  \begin{equation}\label{Cont1/s}
  	\mathfrak{F}(r,x) = \frac{(-1)^r r!}{2 \pi i} \int_{c-i\infty}^{c + i \infty}  x^s \left[ \frac{1}{s-1} \left( \frac{1}{s^{r+1}} + \frac{1}{(s-1)^{r+1}}\right) - \frac{1}{s} \left( \frac{1}{s^{r+1}} + \frac{1}{(s-1)^{r+1}} \right) \right] ds \;.
  \end{equation}
 Similarly denote contribution from $\sum \frac{(-1)^{r+1} r!}{(s-\rho)^{r+1}}$ by $\mathfrak{Z}(r,x)$ and from that of $\tilde{ \Gamma}^{(r)}_K(s)$ by $\mathfrak{G}(r,x)$. Thus we have,
 \begin{equation}\label{Contzero}
 \mathfrak{Z}(r,x)= \frac{(-1)^{r+1} r!}{2 \pi i} \int_{c - i \infty}^{c + \infty} \sum x^s \left[ \frac{1}{(s-1)(s- \rho)^{r+1}} - \frac{1}{s(s-\rho)^{r+1}}\right]ds, \text{ and}
 \end{equation}
 \begin{equation} \label{Contgamma}
 	\mathfrak{G}(r,x) = \frac{x}{2 \pi i} \int_{c - i \infty}^{c + \infty} \frac{x^{s-1}}{s-1 } \tilde{ \Gamma}^{(r)}_K(s) \; ds \; - \;  \frac{1}{2 \pi i} \int_{c - i \infty}^{c + \infty} \frac{x^{s}}{s } \tilde{ \Gamma}^{(r)}_K(s) \; ds.
 \end{equation}
Combining these together we have
 \begin{equation}\label{Contotal}
 	x\Psi_K(1,r,x) - \Psi_K(0,r,x) = 	\mathfrak{F}(r,x) + \mathfrak{Z}(r,x) + \mathfrak{G}(r,x). 
 \end{equation}
 \begin{lm}\label{lem:Mainterm}
 	For $x>1$, consider the recursively defined function
 \begin{align*}
 	f(0,x) &=  \log x,\\
 	f(1,x) &=   \left[ 2+  \frac{1+x}{1-x} \log x +\frac{1}{2}   (\log x )^2 \right] \;\; \text{ and for }r \geq 2  \\
 	f(r,x) &= \frac{1}{r+1} (\log x)^{r+1} \; + \; \frac{1+x}{1-x}(\log x)^{r} + r(r-1)f(r-2,x).
 \end{align*}
 Then $\; \mathfrak{F}(r,x) = (-1)^r (x-1) f(r,x)$.
 \end{lm}
 \begin{proof} We'll proceed by strong induction on $r$. Also, we will repeatedly apply partial fraction decomposition and \eqref{classical11} and \eqref{classical22}. For $r=0$, 
 	\begin{align*}
 		\nonumber \mathfrak{F}(0,x)  & =  \frac{1}{2 \pi i}\int_{c-i \infty}^{c+i \infty}  \;  x^s \left[  \dfrac{1}{s(s-1)}  +  \dfrac{1}{(s-1)^{2}}  - \dfrac{1}{s^{2}} - \dfrac{1}{s(s-1)}\right] \; ds \\[1.2ex]
 		&=(x-1) \log x = (-1)^0 (x-1) f(0,x).
 	\end{align*}
 	For $r=1$, we have 
 	\begin{align*}
 	 \mathfrak{F}(1,x)  & =  \frac{(-1)}{2 \pi i}\int_{c-i \infty}^{c+i \infty}  \;  x^s \left[  \dfrac{1}{s^{2} (s-1)}  +  \dfrac{1}{(s-1)^{3}}  - \dfrac{1}{s^{3}} - \dfrac{1}{s(s-1)^{2}}\right] \; ds \\[1.2ex]
 	\nonumber 	&= (-1)  \left[ x \frac{(\log x)^{2}}{2!} -  \frac{(\log x)^{2}}{2!} \right]  \; \; +  \frac{(-1) }{2 \pi i}\int_{c-i \infty}^{c+i \infty}  \;  x^s \left[  \dfrac{1}{s^{2} (s-1)}  - \dfrac{1}{s(s-1)^{2}}\right] \; ds  \\[1.2ex]
 	\nonumber	& = \frac{(-1)}{2} (x-1) (\log x)^{2} \;\;  +  \;  \frac{(-1)}{2 \pi i}\int_{c-i \infty}^{c+i \infty}  \; x^s \left[  \dfrac{1}{s (s-1)} - \dfrac{1}{s^{2}}  + \dfrac{1}{s(s-1)} - \dfrac{1}{(s-1)^{2}} \right]  \; ds \\[1.2ex] 
 	\nonumber & = \frac{(-1)}{2} (x-1) (\log x)^{2} \; + \; (-1)^{2}(x+1)(\log x) \; + 
  \;  \frac{(-1) 2}{2 \pi i}\int_{c-i \infty}^{c+i \infty}  \; x^s \left[  \dfrac{1}{(s-1)}   - \dfrac{1}{s}  \right]  \; ds  \\[1.2ex]
 	\nonumber	 & =\frac{(-1)}{2} (x-1) (\log x)^{2} \; + \; (-1)^{2}(x+1)(\log x) \; + (-1) 2 (x-1) \\
 	&= (-1)(x-1) \left[ \frac{1}{2}   (\log x )^2 + \frac{1+x}{1-x} \log x +2\right] = (-1)^{1}(x-1) f(1,x).
 	\end{align*}
 	\begin{align*}
 		\nonumber \mathfrak{F}(r,x) & =  \frac{(-1)^r r!}{2 \pi i}\int_{c-i \infty}^{c+i \infty}  \;  x^s \left[  \dfrac{1}{s^{r+1} (s-1)}  +  \dfrac{1}{(s-1)^{r+2}}  - \dfrac{1}{s^{r+2}} - \dfrac{1}{s(s-1)^{r+1}}\right] \; ds \\[1.2ex]  
 		&= (-1)^r r! \left[ x \frac{(\log x)^{r+1}}{(r+1)!} -  \frac{(\log x)^{r+1}}{(r+1)!} \right]  + \frac{(-1)^r r! }{2 \pi i}\int_{c-i \infty}^{c+i \infty}  x^s \left[  \dfrac{1}{s^{r+1} (s-1)}  - \dfrac{1}{s(s-1)^{r+1}}\right]ds  
 			\end{align*}
 		\begin{align*}
 		& = \frac{(-1)^r}{r+1} (x-1) (\log x)^{r+1}  +  \frac{(-1)^r r!}{2 \pi i}\int_{c-i \infty}^{c+i \infty}  \; x^s \left[  \dfrac{1}{s^{r} (s-1)} - \dfrac{1}{s^{r+1}}  + \dfrac{1}{s(s-1)^{r}} - \dfrac{1}{(s-1)^{r+1}} \right] ds \\
 		\nonumber & = \frac{(-1)^r}{r+1} (x-1) (\log x)^{r+1} \; + \; (-1)^{r+1}(x+1)(\log x)^{r} \; + \;\;  \\[1.2ex]
 		\nonumber	& \; \text{\hspace{8cm}} \frac{(-1)^r r!}{2 \pi i}\int_{c-i \infty}^{c+i \infty}  \; x^s \left[  \dfrac{1}{s^{r} (s-1)}   + \dfrac{1}{s(s-1)^{r}}  \right]  \; ds  \\[1.2ex]
 		\nonumber	 & =\frac{(-1)^r}{r+1} (x-1) (\log x)^{r+1} \; + \; (-1)^{r+1}(x+1)(\log x)^{r} \; + \;\;  \\[1.2ex]
 		\nonumber	& \; \text{\hspace{4cm}}   \frac{(-1)^{r} (r)!}{2 \pi i}\int_{c-i \infty}^{c+i \infty}  \; x^s \left[  \dfrac{1}{s^{r-1} (s-1)}  - \dfrac{1}{s^{r}}  - \dfrac{1}{s(s-1)^{r-1}} + \dfrac{1}{(s-1)^r} \right]  ds  \\[1.2ex]
 		\nonumber	& = \frac{(-1)^r}{r+1} (x-1) (\log x)^{r+1} \; + \; (-1)^{r+1}(1+x)(\log x)^{r}+ r(r-1) \mathfrak{F}(r-2,x)  \\[1.2ex]
 			\nonumber & = \frac{(-1)^r}{r+1} (x-1) (\log x)^{r+1} \; + \; (-1)^{r+1}(1+x)(\log x)^{r}+ r(r-1) (-1)^{r-2} (x-1) f(r-2,x) \\[1.2ex]
 		\nonumber	&= (-1)^r (x-1) f(r,x).
 	\end{align*}
 	We note that the second last equality follows from strong induction.
 \end{proof}
 We now focus on the contribution coming from the non-trivial zeros and compute $\mathfrak{Z}(r,x)$ as defined in \eqref{Contzero}. For this, we will do some contour manipulation. As in Figure \ref{fig1} below, for large $T$ and $R$ (to be chosen later), take the contour $C_{R,T} \;$ to be the rectangle : $c-iT \rightarrow c+iT \rightarrow -R+iT \rightarrow -R-iT \rightarrow c-iT$.\\
 \begin{figure}[H]
 	\centering
 	\begin{tikzpicture}[decoration={markings,
 			mark=at position 1cm   with {\arrow[line width=1pt]{stealth}},
 			mark=at position 4.5cm with {\arrow[line width=1pt]{stealth}},
 			mark=at position 7cm   with {\arrow[line width=1pt]{stealth}},
 			mark=at position 9.5cm with {\arrow[line width=1pt]{stealth}}
 		}]
 		\draw[thick, ->] (-6,0) -- (6,0) coordinate (xaxis);
 		
 		\draw[thick, ->] (2,-3) -- (2,3) coordinate (yaxis);
 		
 		\node[above] at (xaxis) {$\mathrm{Re}(s)$};
 		
 		\node[right]  at (yaxis) {$\mathrm{Im}(s)$};
 		
 		\path[draw,Burgundy, line width=0.8pt, postaction=decorate] (4,-2) node[below, Burgundy] {$c-iT$}
 		-- node[midway, right, black] {} (4,2)node[above, Burgundy] {$c+iT$}
 		-- node[midway, right, black] {}(-4,2) node[above, Burgundy] {$-R+iT$};
 		\path[draw,Burgundy, line width=0.8pt, postaction=decorate](-4,2)
 		-- node[midway, left, black] {} (-4,-2) node[below, Burgundy] {$-R-iT$}
 		-- node[midway, above, black] {}(4,-2);
 		\foreach \Point/\PointLabel in {(4,-2)/ , (4,2)/ , (-4,2)/ , (-4,-2)/ , (2,0)/O, (2.5,0)/ \frac{1}{2}, (3,0)/1, (1,0)/-1, (0,0)/-2, (-1,0)/-3, (-2,0)/\cdots }
 		\draw[fill=black] \Point circle (0.05) node[above right] {$\PointLabel$};
 		\draw[densely dotted] (2.5, -3)--(2.5, 3);
 	\end{tikzpicture}
 	\caption{} \label{fig1}
 \end{figure}
 
 We're interested only in the side $c-iT \rightarrow c+iT$. We'll show that for specific choice of $R$, the contribution from the other three sides of the rectangle goes to $0$ as $T \rightarrow \infty$.  
 
 On the side $ \{ -R+iT \rightarrow -R-iT \}$,  writing $\rho = \beta + i \gamma$ and $ \; s = -R + it $  and so, $ds = idt$ we have,
 \begin{align*}
 	\left| \int_{-R+ i T}^{-R - iT} \dfrac{x^{s}}{s} \cdot \dfrac{1}{(s- \rho)^{r+1}} \; ds \right| \;\; & \;\;  \leq \int_{-T}^{T} \dfrac{x^{-R}}{ \sqrt{R^2 + t^2}( ( R+ \beta)^2 + (t-\gamma)^2)^{\frac{r+1}{2}}} \; dt   \\[1.3ex]
 	&\;\; \leq \frac{x^{-R}}{R^{r+2}}  \int_{-T}^{T} dt \leq 2 x^{-R} \frac{T}{R^{r+2}}.
 \end{align*}
 Similarly, for the integral on $\{ c+iT \rightarrow -R+iT \} $, writing $s = \sigma + i T$ 
 \begin{align*}
 	\left| \int_{c+i T}^{-R + iT} \dfrac{x^{s}}{s} \cdot \dfrac{1}{(s- \rho)^{r+1}} \; ds \right| \;\; & \;\;  \leq \int_{c}^{-R} \dfrac{x^{\sigma}}{ \sqrt{\sigma^2 + T^2}( ( \sigma - \beta)^2 + (T-\gamma)^2)^{\frac{r+1}{2}}} \; d \sigma   \\[1.3ex]
 	&\;\; \leq \frac{x^c}{T}  \int_{c}^{-R} \frac{d \sigma}{(\sigma - \beta)^{r+1}} =    \frac{x^c}{T} \left[ \frac{-1}{r(\sigma - \beta)^r}\right]_{c}^{-R} \;\;(\text{ for } r \geq 1) \\[1.3ex]
 	&=\frac{x^c}{rT}   \left[ \frac{(-1)^{r+1}}{(R + \beta)^r} + \frac{1}{(c-\beta)^r}\right]  \\[1.3ex]
 	&\ll \frac{x^c}{rT}.
 \end{align*}
 Note that the last inequality follows from, $0 < \beta < 1$ and we can choose $c\geq2$, so that $c- \beta \geq 1$. Similarly one can show that on the side $\{ -R-iT \rightarrow c-iT\}$, the integral is $O(\frac{x^c}{rT})$.  Thus by choosing $R=T$ and letting, $T \rightarrow \infty$ we see that these integrals go to zero. The computations are similar for the terms of the form $\frac{x^{s}}{(s-1)(s-\rho)^{r+1}}$. Therefore by residue theorem, the line integral on $s=c$ is same as the residue at the poles to the left of $c$.We now compute these residues. From \eqref{Contzero}, we see that the first term of $\mathfrak{Z}(r,x)$ has a simple pole at  $s=1$, and has a pole of order $r+1$ at $s=\rho$, for each $\rho$. Similarly, the second term has a simple pole at $s=0$ and pole of order $r+1$ at $s= \rho$. \\ The residue at $s=0$ is $(-1)^{r+2} r! \sum \dfrac{1}{(- \rho)^{r+1}} = - r! \sum\dfrac{1}{\rho^{r+1}}$. \\ The residue at $s=1$ is  $(-1)^{r+1} r! x \sum \dfrac{1}{(1- \rho)^{r+1}}$.\\ The pole at $\rho$ for the second term has residue 
\begin{align*}
	(-1)^{r+2} r! \;   \lim_{ s \rightarrow \rho} \dfrac{d^r}{ds^r}\left(\dfrac{x^s}{s} \right) 
	& = (-1)^{r+2} r! \;   \lim_{ s \rightarrow \rho} \; \sum_{k=0}^{r} \binom{r}{k} \frac{d^{(r-k)}}{ds^{(r-k)}} x^s \cdot \frac{d^k}{ds^k} \frac{1}{s}\\
	& = (-1)^{r+2} r!  \lim_{ s \rightarrow \rho} \; \sum_{k=0}^{r} \binom{r}{k} x^s (\log x)^{r-k} \frac{(-1)^k k!}{s^{k+1}}\\
	& = (-1)^{r+2} r!  \; \sum_{k=0}^{r} \binom{r}{k} x^{\rho} (\log x)^{r-k} \frac{(-1)^k k!}{\rho^{k+1}}.
\end{align*}

Computations for the first term are very similar and the residue is
\begin{align*}
	&(-1)^{r+1} r! \;   \lim_{ s \rightarrow \rho} \dfrac{d^r}{ds^r}\left(\dfrac{x^s}{s-1} \right)  = (-1)^{r+1} r!  \; \sum_{k=0}^{r} \binom{r}{k} x^{\rho} (\log x)^{r-k} \frac{(-1)^k k!}{(\rho-1)^{k+1}}.
\end{align*}
We summarize these computations in the following lemma.
\begin{lm}\label{lem:Roots} For $r \geq 1$, we have
	$$\lim_{x \rightarrow \infty}\frac{\mathfrak{Z}(r,x)}{(x-1)} = (-1)^{r+1} r!  \sum \frac{1}{(1- \rho)^{r+1}}.   $$
\end{lm}
\begin{proof}
	We note that 
	\begin{align*}
		\mathfrak{Z}(r,x) &=  x r!   (-1)^{r+1}  \sum_{\rho} \frac{1}{(1- \rho)^{r+1}} - \sum_{\rho}\dfrac{r!}{\rho^{r+1}} \\ &+ \sum_{\rho} \left( (-1)^{r+1} \sum_{k=0}^{r} (-1)^k k! \binom{r}{k} x^{\rho} (\log x)^{r-k} \left[ \frac{1}{(\rho-1)^{k+1}} - \frac{1}{\rho^{k+1}} \right] \right)\\[1.2ex]
		&= (x-1) r!  (-1)^{r+1}  \sum_{\rho} \frac{1}{(1- \rho)^{r+1}} + [(-1)^{r+1} - 1]\sum \dfrac{r!}{\rho^{r+1}} \\ &+ \sum_{\rho} \left( (-1)^{r+1} \sum_{k=0}^{r} (-1)^k k! \binom{r}{k} x^{\rho} (\log x)^{r-k} \left[ \frac{1}{(\rho-1)^{k+1}} - \frac{1}{\rho^{k+1}} \right] \right)\\[1.2ex]
		\frac{\mathfrak{Z}(r,x)}{x-1} & = r!  (-1)^{r+1}  \sum_{\rho} \frac{1}{(1- \rho)^{r+1}} + \frac{[(-1)^{r+1} - 1]}{(x-1)}\sum \dfrac{r!}{\rho^{r+1}} \\ &+\sum_{\rho} \left( (-1)^{r+1} \sum_{k=0}^{r} (-1)^k k! \binom{r}{k} \left[ \frac{1}{(\rho-1)^{k+1}} - \frac{1}{\rho^{k+1}} \right] \frac{ x^{\rho} (\log x)^{r-k}}{(x-1)}\right).
	\end{align*}
	We first look at the second sum.  For $r$ odd, this term is zero anyway. For $r \geq 1$ even, the series is absolutely convergent and so as $x \rightarrow \infty$, this term goes to $0$. If we write $\rho = \beta + i \gamma$, then note that the last sum, excluding finitely many low hanging $\rho$'s,  is $O \left( 2^r r! (\log x)^r \sum_{\rho}\dfrac{x^{\beta-1}}{\gamma^2} \right)$. Since there is $(x-1)$ in the denominator, and $x \rightarrow \infty$,   excluding finitely many $\rho$'s does not have any effect, thus we only need to look at the sum $\sum_{\rho} \frac{x^{\beta - 1}}{\gamma^2}$, for $|\gamma|$ sufficiently large. \\
	
	We make use of standard zero-free regions of the Dedekind zeta function. In particular, we will use the following of \cite[Lemma 8.1]{lagarias1} which states that there is an absolute, effectively computable positive constant $c$ such that $\z_K(s)$ has no zeros $\rho = \beta + i \gamma$ in the region $$ |\gamma| \geq \frac{1}{1 + 4 \log |d_K|} \;\;, \;\;\;\; \beta \geq 1-\frac{c}{\log |d_K| + n_K \log(|\gamma| + 2)}.$$
Thus for $|\gamma|$ sufficiently large,
	\begin{align}
		\nonumber & \sum \dfrac{ \; x^{\beta - 1}}{\gamma^2} < \sum  \dfrac{ \; x^{-c (\log |d_K| + n_K \log(|\gamma| + 2))^{-1}}}{\gamma^2}
	= \sum_{ \log |d_K| + n_K \log(|\gamma| + 2) \; < \; T} \;\; + \sum_{ \log |d_K| + n_K \log(|\gamma| + 2) \; \geq \; T}
	\end{align}
	where we will choose $T = \sqrt{\log x} \;$. Thus for the first sum, we have 
	\begin{align*}
		& \sum_{ \log |d_K| + n_K \log(|\gamma| + 2) \; < \; T} \dfrac{ x^{-c (\log |d_K| + n_K \log(|\gamma| + 2))^{-1}}}{\gamma^2}   <  \; \sum  \dfrac{ x^{-cT^{-1}}}{\gamma^2}	=\left(\sum \dfrac{1}{\gamma^2}\right) \exp({-c \sqrt{\log x}}).
	\end{align*}
	Note that the last equality follows from $$ \exp({-c \sqrt{\log x}}) = \exp(-c \log x (\sqrt{\log x})^{-1}) = \exp(\log x^{-cT^{-1}}) = x^{-cT^{-1}}$$
	Now as $x \rightarrow \infty$, clearly $\exp({-c \sqrt{\log x}}) \rightarrow 0$. We also note that 
	$$\lim_{x \rightarrow \infty} \dfrac{(\log x)^r}{e^{c \sqrt{\log x}}} = \lim_{y \rightarrow \infty} \dfrac{y^{2r}}{ e^{cy}} \; \rightarrow \; 0.$$
	$\;$  \\
	Now for the second sum,
	$$ \log |d_K| + n_K \log(|\gamma| + 2) \; \geq \; \sqrt{\log x} \; \Rightarrow \; |\gamma| \geq -2 + \exp \left(\frac{\sqrt{\log x} - \log |d_K|}{n_K} \right).
	$$ We will write the expression on the right as $u$, i.e. $|\gamma| \geq u$. Note that as $x \rightarrow \infty $, so does $u$. \\  
	Since $x>1$ we have
	\begin{align*}
		 \sum_{ |\gamma| \geq u} \dfrac{ \; x^{-c (\log |d_K| + n_K \log(|\gamma| + 2))^{-1}}}{\gamma^2} & <  \;\; \left( \sum_{|\gamma| \geq u }  \dfrac{1 }{\gamma^2}	\right)  \leq  \left( \sum_{j>u} \; \sum_{j < |\gamma| < j+1} \dfrac{1 }{j^2}	\right)  \\[1.3ex] 
		& \leq \; \sum_{j>u} \frac{c_1 (n_K \log j + \log |d_K|)}{j^2} \;\;\; \text{ (by Lemma \ref{numZero})}\\[1.3ex]
		& \leq c_1 n_K \dfrac{\log u +1}{u} + c_1 (\log |d_K|) \frac{1}{u}
			\end{align*}
		\begin{align*}
		\text{Thus } \;\;\; (\log x)^r  \sum_{ |\gamma| \geq u} \dfrac{ \; x^{-c (\log |d_K| + n_K \log(|\gamma| + 2))^{-1}}}{\gamma^2}  & < c_1 n_K (\log x)^r \dfrac{\log u +1}{u} + c_1 (\log |d_K|)(\log x)^r \frac{1}{u}\\
		& \rightarrow 0  \; \text{ as } \; u \rightarrow \infty \;. 
	\end{align*}
	The above is true since, 
	\begin{align*}
		u &= -2 + \exp \left( \frac{\sqrt{\log x} - \log |d_K|}{n_K}\right) \Rightarrow
		\log(u+2) =\frac{\sqrt{\log x} - \log |d_K|}{n_K}\\
	&\text{and so, } \;	\log x = (n_K \log(u+2) + \log |d_K| )^2 
	\end{align*}
	This completes the proof.
	\end{proof}
	Lastly we compute $\mathfrak{G}(r,x)$. We first note that, using the series expansion of  \eqref{gammaformula2} we have 
\begin{align}\label{Gammakm}
	\nonumber \tilde{\Gamma}_K^{(r)}(s) &= - \frac{r_1}{2^{r+1}} \sum_{k=0}^{\infty} \frac{(-1)^{r} r! \; 2^{r+1}}{(s+2k)^{r+1}} - r_2 \sum_{k=0}^{\infty} \frac{(-1)^{r} r!}{(s+k)^{r+1}}\\[1.2ex]
	&= (-1)^{r+1} r! \left[ \frac{(r_1 + r_2)}{s^{r+1}}  + r_1  \sum_{k = 1}^{\infty}  \frac{1}{(s+2k)^{r+1}} + r_2  \sum_{k = 1}^{\infty}  \frac{1}{(s+k)^{r+1}} \right] .
\end{align} 
In the last line, we have separated the $k=0$ case, and the sums now start from $k=1$.   We have the partial fraction decomposition 
$$\frac{1}{s^{r+1}(s-1)} = \frac{1}{s-1} - \frac{1}{s} - \frac{1}{s^2} - \cdots  - \frac{1}{s^{r+1}}.$$
Recall, we are trying to compute the integral
$$
	\mathfrak{G}(r,x) = \frac{x}{2 \pi i} \int_{c - i \infty}^{c + \infty} \frac{x^{s-1}}{s-1 } \tilde{ \Gamma}^{(r)}_K(s) \; ds \; - \;  \frac{1}{2 \pi i} \int_{c - i \infty}^{c + \infty} \frac{x^{s}}{s } \tilde{ \Gamma}^{(r)}_K(s) \; ds.
$$
Contribution of the first term in \eqref{Gammakm} is
\begin{equation}\label{gamma_m_main}
	(-1)^{r+1}  r! (r_1+r_2) \left[ x-1 - \sum_{k=1}^{r+1} \frac{1}{(k)!} (\log x)^{k}  \right].
\end{equation}
We have used the classical formula (\ref{classical22}) multiple times to get the above.
Now to compute the contribution from the series terms in (\ref{Gammakm}), we first notice the similarity of it to the sum on non-trivial zeros $\sum \frac{(-1)^r r!}{(s- \rho)^{r+1}}$. The residue computations are very similar, with $\rho$ replaced by $-2k$ or $-k$. For the first series, we are computing the integral
$$(-1)^{r+1} r! \; r_1 \left[ \frac{1}{2 \pi i} \int_{c - i \infty}^{c + i \infty} x^{s} \left( \sum_{k = 1}^{\infty}  \frac{1}{(s-1)(s+2k)^{r+1}}  \; -    \frac{1}{s(s+2k)^{r+1}} \right) \; ds  \right]. $$
In the second term, the pole at $s=0$ has residue $ \sum_{k=1}^{\infty} \dfrac{(-1)^{r+2} r! r_1 }{(2k)^{r+1}}$.  \\The pole at $s = -2k$ ( order = $r+1$) has residue 
\begin{align*}
	(-1)^{r+2} r! r_1\;   \lim_{ s \rightarrow (-2k) } \dfrac{d^r}{ds^r}\left(\dfrac{x^s}{s} \right)  &  = (-1)^{r+2} r!  r_1 \;   \lim_{ s \rightarrow (-2k)} \; \sum_{j=0}^{r} \binom{r}{j} \frac{d^{(r-j)}}{ds^{(r-j)}} x^s \cdot \frac{d^j}{ds^j} \frac{1}{s}\\
	&; = (-1)^{r+2} r!  r_1 \; \sum_{j=0}^{r} \binom{r}{j} x^{-2k} (\log x)^{r-j} \frac{(-1)^j j!}{(-2k)^{j+1}}
\end{align*}
Computations for the other terms are very similar. Thus the net contribution from the first series of equation (\ref{Gammakm}) is given by
\begin{align}\label{GammaKm2}
	\nonumber  &	(-1)^{r+1} r! \; r_1 \left[  x \sum_{k=1}^{\infty} \frac{1}{(1+2k)^{r+1}} - \sum_{k=1}^{\infty} \frac{1}{(2k)^{r+1}} \right] + \\[1.3ex] 
	& (-1)^{r+1} r! \; r_1  \sum_{k=1}^{\infty} \left[ \sum_{j=0}^{r} \binom{r}{j} x^{-2k} (\log x)^{r-j}  j! \left( \frac{1}{(2k)^{j+1}} - \frac{1}{(2k+1)^{j+1}}\right) \right]
\end{align}
Looking at this, we can precisely write down the contribution from the second series of (\ref{Gammakm}), by just replacing $2k$ by $k$, and $r_1$ by $r_2$.   We get,
\begin{align}\label{GammaKm3}
	\nonumber  &	(-1)^{r+1} r! \; r_2 \left[  x \sum_{k=1}^{\infty} \frac{1}{(1+k)^{r+1}} - \sum_{k=1}^{\infty} \frac{1}{(k)^{r+1}} \right] + \\[1.3ex] 
	\nonumber	& (-1)^{r+1} r! \; r_2  \sum_{k=1}^{\infty} \left[ \sum_{j=0}^{r} \binom{r}{j} x^{-k} (\log x)^{r-j}  j! \left( \frac{1}{(k)^{j+1}} - \frac{1}{(k+1)^{j+1}}\right) \right] \\[1.3ex] 
	\nonumber & = (-1)^{r+1} r! \; r_2 (x-1) \sum_{k=1}^{\infty} \frac{1}{(1+k)^{r+1}} \; + \;  (-1)^{r} r! \; r_2 \; + \\[1.3ex] 
	& (-1)^{r+1} r! \; r_2  \sum_{k=1}^{\infty} \left[ \sum_{j=0}^{r} \binom{r}{j} x^{-k} (\log x)^{r-j}  j! \left( \frac{1}{(k)^{j+1}} - \frac{1}{(k+1)^{j+1}}\right) \right]
\end{align}
As before summarize these computations to get the following lemma.
\begin{lm}\label{lem:Gamma} For $r \geq 1$ we have,
	$$\lim_{x \rightarrow \infty} \frac{\mathfrak{G}(r,x)}{(x-1)} = \tilde{\Gamma}_K^{(r)}(1)$$ 
\end{lm}
\begin{proof}
	We see from \eqref{gamma_m_main}, \eqref{GammaKm2} and \eqref{GammaKm3}
	\begin{align*}
		\frac{\mathfrak{G}(r,x)}{(x-1)} &= (-1)^{r+1} r! (r_1 + r_2)  + (-1)^{r+1} r! \; r_1 \sum_{k=1}^{\infty} \frac{1}{(1+2k)^{r+1}} \\ 
		&+ (-1)^{r+1} r! \; r_2  \sum_{k=1}^{\infty} \frac{1}{(1+k)^{r+1}} + O_{r} \left( \frac{n_K (\log x)^{r+1}}{x} \right).
	\end{align*}
	Here we are using the notation $O_{r}(.)$ to mean the implicit constant depends on $r$. Thus,
	\begin{align*}
	\lim_{x \rightarrow \infty}	\frac{\mathfrak{G}(r,x)}{(x-1)} &= (-1)^{r+1} r! (r_1 + r_2)  + (-1)^{r+1} r! \; r_1 \sum_{k=1}^{\infty} \frac{1}{(1+2k)^{r+1}} \\ 
		&+ (-1)^{r+1} r! \; r_2  \sum_{k=1}^{\infty} \frac{1}{(1+k)^{r+1}}\\
		& = (-1)^{r+1} r! \left[ r_1 \sum_{k=0}^{\infty} \frac{1}{(1+2k)^{r+1}} +  r_2  \sum_{k=0}^{\infty} \frac{1}{(1+k)^{r+1}} \right] \\
		& = \tilde{\Gamma}_K^{(r)}(1).
	\end{align*}
\end{proof}
\subsection{Proof of Theorem \ref{thm:EKgenIhara}}
From \eqref{PhiEuler} and \eqref{Contotal} we get,
\begin{align*}
	(x-1)(-1)^r \Phi_K(r,x) &= \mathfrak{F}(r,x) + \mathfrak{Z}(r,x) + \mathfrak{G}(r,x) \\
	(-1)^r \Phi_K(r,x) & = \frac{ \mathfrak{F}(r,x)}{(x-1)} + \frac{ \mathfrak{Z}(r,x)}{(x-1)} + \frac{ \mathfrak{G}(r,x)}{(x-1)}
\end{align*}
Now by Lemma \ref{lem:Mainterm}, we have 
\begin{align*}
	(-1)^r \left[ \Phi_K(r,x) - f(r,x) \right] = \frac{ \mathfrak{Z}(r,x)}{(x-1)} + \frac{ \mathfrak{G}(r,x)}{(x-1)}
\end{align*}
Taking limit $x \rightarrow \infty$ and using Lemma \ref{lem:Roots} and Lemma \ref{lem:Gamma} we get, for $r \geq 1$,
\begin{align*}
	\lim_{x \rightarrow \infty}(-1)^r \left[ \Phi_K(r,x) - f(r,x) \right] &= (-1)^{r+1} r!  \sum \frac{1}{(1- \rho)^{r+1}}  + \tilde{\Gamma}_K^{(r)}(1).
\end{align*}
And so, by \eqref{gamma_kr}
\begin{align*}
	\lim_{x \rightarrow \infty}(-1)^r \left[ \Phi_K(r,x) - f(r,x) \right]  &= r! [ (-1)^{r+1} - \gamma_{K,r}] \\
	\gamma_{K,r} + (-1)^{r} &= \frac{(-1)^{r+1}}{r!} \lim_{x \rightarrow \infty} \left( \frac{}{} \Phi_K(r,x) - f(r,x) \right).  
\end{align*}
Combining this, with the known $r=0$ case due to Ihara, as mentioned in \eqref{Ihara_lim_form}, completes the proof. \hspace{13.8cm} $\qed$
 \newpage
\bibliographystyle{amsplain}
\bibliography{EulerKron}
\end{document}